\documentclass[12pt]{article}
\usepackage[top=1in,bottom=1in,left=1in,right=1in]{geometry}
\usepackage{indentfirst}
\usepackage{amsfonts,amsmath,amsthm,amssymb,hyperref,enumerate,bm}
\usepackage{longtable}
\usepackage[caption=false]{subfig}
\usepackage{leftidx}
\usepackage{lineno}
\usepackage{color,xcolor}
\usepackage{array}
\usepackage{latexsym}
\usepackage{float}
\usepackage{color}
\usepackage{longtable,supertabular,booktabs}
\usepackage{rotating}
\usepackage{cases}
\usepackage{multicol,multirow}
\usepackage{epsfig}
\usepackage{cite}
\usepackage{fancyhdr}       
\usepackage{dsfont}
\usepackage{hyperref}
\hypersetup{
	colorlinks=true,
	linkcolor=red,
	filecolor=blue,
	urlcolor=red,
	citecolor=blue,
}

\newtheorem{theorem}{Theorem}[section]

\newtheorem{lemma}{Lemma}[section]

\newtheorem{corollary}{Corollary}[section]
\newtheorem{remark}{Remark}[section]
\newtheorem{definition}{Definition}[section]

\newtheorem{example}{Example}[section]

\newtheorem{problem}{Problem}[section]

\def\G1{G^\mathcal{C}}

\begin{document}
\title{Permutation polynomials of the form $x+\gamma \mathrm{Tr}(H(x))$}
\author{
Yangcheng Li$^{1,}$\footnote{Corresponding author. E-mail\,$:$ liyc@m.scnu.edu.cn.} \,\, Xuan Pang$^{1,}$\footnote{E-mail\,$:$ pangxuan202503@163.com.} \,\, Pingzhi Yuan$^{1,}$\footnote{E-mail\,$:$ yuanpz@scnu.edu.cn. Supported by the National Natural Science Foundation of China (Grant No. 12171163) and Guangdong Basic and Applied Basic Research Foundation (Grant No. 2024A1515010589).} \,\, Yuanpeng Zeng$^{1,}$\footnote{E-mail\,$:$ zengyp2025@163.com.}   \\
{\small\it  $^{1}$School of Mathematical Sciences, South China Normal University,}\\
{\small\it Guangzhou 510631, Guangdong, P. R. China} \\
}

\date{}
\maketitle
\date{}

\noindent{\bf Abstract}\quad
Given a polynomial \( H(x) \) over \(\mathbb{F}_{q^n}\), we study permutation polynomials of the form \( x + \gamma \mathrm{Tr}(H(x)) \) over \(\mathbb{F}_{q^n}\). Let
\[P_H=\{\gamma\in \mathbb{F}_{q^n} : x+\gamma \mathrm{Tr}(H(x))~\text{is a permutation polynomial}\}.\]
We present some properties of the set \(P_H\), particularly its relationship with linear translators. Moreover, we obtain an effective upper bound for the cardinality of the set \(P_H\) and show that the upper bound can reach up to $q^n - q^{n - 1}$. Furthermore, we prove that when the cardinality of the set \(P_H\) reaches this upper bound, the function \(\mathrm{Tr}(H(x))\) must be an \(\mathbb{F}_q\)-linear function. Finally, we study two classes of functions $H(x)$ over \(\mathbb{F}_{q^2}\) and determine the corresponding sets $P_H$. The sizes of these sets $P_H$ are all relatively small, even only including the trivial case.


\medskip
\noindent{\bf MR(2020) Subject Classification} 11T06, 11T55.

\section{Introduction}
A finite field, denoted as $\mathbb{F}_{q^n}$, where $q$ is a prime power and $n$ is a positive integer, is a field that contains a finite number of elements. Denote $\mathbb{F}_{q^n}[x]$ as the ring of polynomials in a single indeterminate $x$ over $\mathbb{F}_{q^n}$. A polynomial $f(x)$ over a finite field $\mathbb{F}_{q^n}$ is called a {\it permutation polynomial} if the mapping $f : x\mapsto f(x)$ from $\mathbb{F}_{q^n}$ to $\mathbb{F}_{q^n}$ is a bijection. In other words, for every $y\in\mathbb{F}_{q^n}$, there exists a unique $x\in\mathbb{F}_{q^n}$ such that $f(x)=y$.

Permutation polynomials over finite fields have attracted significant attention due to their wide-ranging applications, especially in coding theory \cite{Ding,Ding-Zhou,Laigle-Chapuy}, cryptography \cite{Rivest-Shamir-Adelman,Schwenk-Huber}, combinatorial design theory \cite{Ding-Yuan}, and other areas of mathematics and engineering \cite{Lidl-Niederreiter2,Mullen}. Constructing a class of permutation polynomials with a simple form or determining whether a class of polynomials is a permutation polynomial is an interesting and challenging problem, and there are a great many results in this regard \cite{Li-Qu-Wang,Tuxanidy-Wang,Wu-Yuan,Yuan1,Yuan2,Yuan-Ding1,Yuan-Ding2,Yuan-Zeng}.

In recent years, polynomials of the form $x+\gamma\text{Tr}(H(x))$ over $\mathbb{F}_{q^n}$ have emerged as an important class of polynomials to study. Here, $\gamma\in\mathbb{F}_{q^n}$, $\mathrm{Tr}(\cdot)$ represents the trace function which maps an element from $\mathbb{F}_{q^n}$ to $\mathbb{F}_{q}$, and $H(x)$ is a polynomial over $\mathbb{F}_{q^n}$. The definition of $\mathrm{Tr}(\cdot)$ is as follows: for an element $x\in\mathbb{F}_{q^n}$,
\[\mathrm{Tr}(x)=x + x^q+\cdots+x^{q^{n-1}}.\]
The study of when such polynomials are permutation polynomials has both theoretical and practical significance. The main motivation for studying the polynomial \(x+\gamma\mathrm{Tr}(H(x))\) is to consider whether a permutation polynomial remains a permutation polynomial after a slight modification. For example, if both \(F(x)\) and \(F(x) + x\) are permutation polynomials, then \(F(x)\) is called a {\it complete mapping polynomial}. For results related to complete mapping polynomials we can refer to \cite{Laigle-Chapuy,Niederreiter-Robinson}. Given a polynomial \(H(x)\), determining for which \(\gamma\in\mathbb{F}_{q^n}\) the polynomial \(x+\gamma\text{Tr}(H(x))\) can be a permutation polynomial is a challenging problem.

In 2009, Charpin and Kyureghyan \cite{Charpin-Kyureghyan1} considered polynomials of the form
\begin{equation}
G(x)+\gamma\mathrm{Tr}(H(x))
\end{equation}
and introduced an effective method to construct many such permutations. In this paper, they also proposed an open problem: Characterize a class of permutation polynomials of type (1), where $G(X)$ is neither a permutation nor linearized polynomial. In 2025, Yuan \cite{Yuan3} proposed a new algebraic structure of permutation polynomials over \(\mathbb{F}_{q^n}\) and gave an answer to the above open problem.

In 2010, Charpin and Kyureghyan \cite{Charpin-Kyureghyan2} further studied some properties of the polynomial \(x^s+\gamma\text{Tr}(x^t)\) over \(\mathbb{F}_{2^n}\) (also see \cite{Charpin-Kyureghyan-Suder}). In 2016, Kyureghyan and Zieve \cite{Kyureghyan-Zieve} searched all permutation polynomials over fields $\mathbb{F}_{q^n}$ of the form $X+\gamma\mathrm{Tr}(X^k)$ with $\gamma\in\mathbb{F}_{q^n}$, $q$ odd, $n > 1$, and $q^n < 5000$. In 2018, Li et al. \cite{Li-Qu-Chen-Li} presented fifteen new classes of permutation polynomials of the form $cx + \mathrm{Tr}_{q^l/q}(x^a)$ over finite fields with even characteristic. In 2019, Zha, Hu and Zhang \cite{Zha-Hu-Zhang} continued the work of \cite{Kyureghyan-Zieve} and studied permutation polynomials of the form \(x + \gamma\mathrm{Tr}_{q^n/q}(h(x))\) over \(\mathbb{F}_{q^n}\) for some values of \(p\), \(k\) and \(n\). In 2025, Jiang et al. \cite{Jiang-Yuan-Li-Qu} studied permutation polynomials of the form \(x + \gamma\mathrm{Tr}_{q^2/q}(h(x))\) over finite fields with even characteristic, where \(h(x)\) represents certain specific polynomials. For more results, we refer to \cite{Wang-Zha-Du-Zheng}.

In 1992, Evans, Greene and Niederreiter \cite{Evans-Greene-Niederreiter} studied a related problem. They considered the polynomial \(f\in \mathbb{F}_{q}[x]\) with \(\deg(f) < q\). When \(f(x)+cx\) is a permutation polynomial for at least \([q/2]\) values of \(c\in \mathbb{F}_{q}\), they proved several properties. For example, \(f(x)+cx\) is a permutation polynomial for at least \(q-\frac{q - 1}{p - 1}\) values of \(c\in \mathbb{F}_{q}\), and \(f(x)\) can be expressed as \(f(x)=ax + g(x^{p})\), where \(a\in F_{q}\) and \(g\in \mathbb{F}_{q}[x]\).

Let \(f\in\mathbb{F}_q[x]\). Define
\[D_f := \left\{ \frac{f(y) - f(x)}{y - x} \,\middle|\, x, y \in \mathbb{F}_q, x\neq y \right\}\]
as the set of directions determined by the function \(f\) and $N=|D_f|$. In 1970, R\'edei \cite{Redei} proved that either \( N = 1 \), \( N \geq (q + 3)/2 \) or \( (q + s)/(s + 1) \leq N \leq (q - 1)/(s - 1) \) for some \( s = p^e \leq \sqrt{q} \). In \cite{Blokhuis-Brouwer-Szonyi,Blokhuis-Ball-Brouwer-Storme-Szonyi}, some results were improved. In 2003, Ball \cite{Ball} covered the unresolved cases, and showed that either \( N = 1 \), \( N \geq (q + 3)/2 \) or \( q/s + 1 \leq N \leq (q - 1)/(s - 1) \) for some \( s \) where \( \mathbb{F}_s \) is a subfield of \( \mathbb{F}_q\). Moreover, Ball \cite{Ball} proved that if \(|D_f| \leq \frac{q + 1}{2}\), then \(f\) is a linearized polynomial. The bounds given in \cite{Ball} can all be attained, and we refer to the results in \cite{Gacs-Lovasz-Szonyi,Lovasz-Schrijver,Polverino-Szonyi-Weiner,Redei}.

In this paper, we investigate for a given polynomial $H(x)$ over $\mathbb{F}_{q^n}$, what values of $\gamma$ will make $x + \gamma\mathrm{Tr}(H(x))$ a permutation polynomial over $\mathbb{F}_{q^n}$. In other words, we consider the following set
\[P_H=\{\gamma\in \mathbb{F}_{q^n} : x+\gamma \mathrm{Tr}(H(x))~\text{is a permutation polynomial}\}.\]
First, we give an equivalent characterization of the set $P_H$ and present some properties of the set $P_H$, particularly its relationship with linear translators. Then, we provide an effective upper bound for the cardinality of the set $P_H$ and show that the upper bound can reach up to $q^n - q^{n - 1}$. Furthermore, we prove that when the cardinality of the set \(P_H\) reaches this upper bound, the function \(\mathrm{Tr}(H(x))\) must be an \(\mathbb{F}_q\)-linear function. Finally, we study two classes of functions $H(x)$ over \(\mathbb{F}_{q^2}\) and determine the corresponding sets $P_H$. The sizes of these sets $P_H$ are all relatively small, even only including the trivial case.

\section{Some properties of the set \( P_H \)  }
We first give an equivalent characterization of the set \(P_H\). Let
\[I_b=\{x\in \mathbb{F}_{q^n} : \mathrm{Tr}(H(x))=b\}\]
and
\[S=\mathrm{Tr}(H(\mathbb{F}_{q^n})).\]
We have

\begin{theorem} The set $P_H=\mathbb{F}_{q^n} \setminus \bigcup_{\substack{b,c \in S\\b\neq c}}\left\{\frac{y-x}{b-c} : x\in I_b,y\in I_c\right\}.$
\end{theorem}

\begin{proof}
The polynomial \(x + \gamma \mathrm{Tr}(H(x))\) is a permutation polynomial if and only if, for \(x\neq y\), \(x + \gamma \mathrm{Tr}(H(x))\neq y + \gamma \mathrm{Tr}(H(y))\). If \( x \in I_b \) and \( y \in I_b \), then \( \mathrm{Tr}(H(x)) = \mathrm{Tr}(H(x)) \). Therefore, \( x + \gamma \mathrm{Tr}(H(x)) \neq y + \gamma \mathrm{Tr}(H(y)) \). When considering \(x\in I_b\), \(y\in I_c\) with \(b\neq c\), this condition is equivalent to \(x+\gamma b\neq y+\gamma c\). Rearranging this inequality leads to \(\gamma\neq\frac{y - x}{b - c}\) for all \(x\in I_b\), \(y\in I_c\) and \(b\neq c\). By the definition of the set \(P_H\), this is equivalent to \(\gamma\in P_H\). Thus, \(x + \gamma \mathrm{Tr}(H(x))\) is a permutation polynomial if and only if \(\gamma\in P_H\).
\qed\end{proof}

\begin{remark} Let \( \hat{P}_H =\bigcup_{\substack{b,c \in S\\b\neq c}}\left\{\frac{y-x}{b-c} : x\in I_b,y\in I_c\right\} \), then \( P_H \cup \hat{P}_H = \mathbb{F}_{q^n} \), and \( 0 \notin \hat{P}_H\). It is easy to see that \( |P_H| + |\hat{P}_H| = |\mathbb{F}_{q^n}| = q^n \). Let \( D_f \) be the direction set determined by the function \( f \). The set \( \hat{P}_H \) is closely related to \( D_{\mathrm{Tr}(H(x))} \). Let \( x \in I_b \) and \( y \in I_c \) with $b\neq c$. Suppose \( d = \frac{Tr(H(x)) - Tr(H(y))}{x - y}= \frac{b - c}{x - y} \) and \( d \neq 0 \). Then we have \( -\frac{1}{d} = \frac{y - x}{b - c} \in \hat{P}_H \). This deduces that \( |D_{\mathrm{Tr}(H(x))}| - 1 = |\hat{P}_H| \). Therefore, we get \( |D_{\mathrm{Tr}(H(x))}| + |P_H| = |\mathbb{F}_{q^n}| + 1 = q^n + 1 \).
\end{remark}

\begin{theorem}
$P_H=\mathbb{F}_{q^n}$ if and only if $\mathrm{Tr}(H(x))$ is a constant map.
\end{theorem}

\begin{proof}
Suppose \(\mathrm{Tr}(H(x))\) is a constant map. Let \(\mathrm{Tr}(H(x)) = k\) for all \(x\in\mathbb{F}_{q^n}\). Then the polynomial \(x+\gamma\mathrm{Tr}(H(x))\) simplifies to \(x+\gamma k\). The polynomial \(x + \gamma k\) is a permutation polynomial for all \(\gamma\in\mathbb{F}_{q^n}\). Therefore, \(P_H=\mathbb{F}_{q^n}\).

Assume \(P_H=\mathbb{F}_{q^n}\). According to Theorem 2.1, the set \(\hat{P}_H=\bigcup_{\substack{b,c \in S\\b\neq c}}\left\{\frac{y - x}{b - c}:x\in I_b,y\in I_c\right\}\) must be the empty set. The only way this can happen is when the set \(S\), which is the range of \(\mathrm{Tr}(H(x))\) in \(\mathbb{F}_{q^n}\), has only one element. In other words, \(\mathrm{Tr}(H(x))\) takes a single value for all \(x\in\mathbb{F}_{q^n}\), so \(\mathrm{Tr}(H(x))\) is a constant map.
\qed\end{proof}

According to Theorem 2.2, when \(\mathrm{Tr}(H(x))\) is a constant map, it corresponds to a trivial case. Therefore, in the following, we only consider the case where \(\mathrm{Tr}(H(x))\) is not a constant map.

If we impose a certain restriction on the polynomial \(H(x)\), we can also obtain the following theorem.

\begin{theorem}
If $H(x)\in\mathbb{F}_{p^t}[x]$, then $\alpha\notin P_H$ if and only if $\alpha^{p^t}\notin P_H$.
\end{theorem}

\begin{proof}
If \(\alpha\notin P_H\), then by Theorem 2.1, there exist \(x\in I_b\) and \(y\in I_c\) with $b\neq c$ such that \(\alpha=\frac{y - x}{b - c}\), where \(\mathrm{Tr}(H(x)) = b\) and \(\mathrm{Tr}(H(y)) = c\). Since \(H(x)\in\mathbb{F}_{p^t}[x]\), we have
\[b^{p^t}=\mathrm{Tr}^{p^t}(H(x))=\mathrm{Tr}(H(x^{p^t})),\quad c^{p^t}=\mathrm{Tr}^{p^t}(H(y))=\mathrm{Tr}(H(y^{p^t})),\]
and clearly \(b^{p^t}\neq c^{p^t}\). Then,
\[\alpha^{p^t}=\left(\frac{y-x}{b-c}\right)^{p^t}=\frac{y^{p^t}-x^{p^t}}{b^{p^t}-c^{p^t}}.\]
By Theorem 2.1, we can conclude that \(\alpha^{p^t}\notin P_H\).

Conversely, if $\alpha^{p^t}\notin P_H$, then $\alpha^{p^t},\alpha^{p^{2t}},\ldots,\alpha^{q^n}=\alpha$ are all not in $P_H$. Thus, the theorem is proved.
\qed\end{proof}

\subsection{The relationship between $P_H$ and linear translators}
In this subsection, we illustrate the relationship between the set \(P_H\) and linear translators. The definition of linear translators is given as follows.
\begin{definition}[\hspace{-0.01em}\cite{Kyureghyan}]
Let \(f: \mathbb{F}_{q^n} \to \mathbb{F}_q\), \(a \in \mathbb{F}_q\) and \(\alpha\) be a nonzero element in \(\mathbb{F}_{q^n}\). If \(f(x + u\alpha) - f(x) = ua\) for all \(x \in \mathbb{F}_{q^n}\) and all \(u \in \mathbb{F}_q\), then we say that \(\alpha\) is an \(a\)-linear translator of the function \(f\).
\end{definition}

In 2008, Charpin and Kyureghyan \cite{Charpin-Kyureghyan0} proved that if \( G(X) \) is a permutation polynomial or a linearized polynomial, then the problem of studying permutation polynomials over \(\mathbb{F}_{2^n}\) of the form \( G(X) + \gamma \mathrm{Tr}(H(X)) \) reduces to finding Boolean functions with linear structures.

We consider the general case over $\mathbb{F}_{q^n}$ and obtain the following theorems.

\begin{theorem}
$\alpha\mathbb{F}_q\subset P_H$ if and only if $\alpha$ is a $0$-linear translator of $\mathrm{Tr}(H(x)).$
\end{theorem}

\begin{proof}
The element \(\alpha\) is a \(0\)-linear translator of \(\mathrm{Tr}(H(x))\) if and only if \(\mathrm{Tr}(H(x+\alpha u)) = \mathrm{Tr}(H(x)) = b_x\) for all \(u \in \mathbb{F}_q\). This is equivalent to stating that for every \(x \in \mathbb{F}_{q^n}\), there exists some \(b_x \in S\) such that the coset \(x + \alpha\mathbb{F}_q\) is contained in \(I_{b_x}\). We now show that this last condition is equivalent to the set inclusion \(\alpha\mathbb{F}_q \subset P_H\).

Assume that \(\alpha\mathbb{F}_q \subset P_H\). We proceed by contradiction: suppose there exist \(a_1 \neq a_2\) in \(\mathbb{F}_q\) satisfying \(x + \alpha a_1 \in I_b\) and \(x + \alpha a_2 \in I_c\) for some \(b \neq c\) in \(S\). By Theorem 2.1, it follows that
\[\frac{(x + \alpha a_1) - (x + \alpha a_2)}{b - c} = \frac{\alpha(a_1 - a_2)}{b - c} \notin P_H,\]
which contradicts the fact that \(\alpha \cdot \frac{a_1 - a_2}{b - c} \in \alpha\mathbb{F}_q \subset P_H\).

Conversely, assume that \(\alpha\mathbb{F}_q \not\subset P_H\). Then there exists \(u \in \mathbb{F}_q\) such that \(\alpha u \notin P_H\). By the definition of \(P_H\), there exist \(x \in I_b\), \(y \in I_c\) with \(b \neq c\) satisfying \(\alpha u = \frac{y - x}{b - c}\). This implies \(y = x + \alpha u(b - c) \in x + \alpha\mathbb{F}_q\subset I_b\), it follows that \(y \in I_b\), contradicting \(y \in I_c\) with \(b \neq c\).
\qed\end{proof}

Now, let us explore the case where \(\mathrm{Tr}(H(x))\) has a $b$-linear translator, where \(b \neq 0\). Recall that \(I_b=\{x\in \mathbb{F}_{q^n} : \mathrm{Tr}(H(x))=b\}\).

\begin{theorem}
Suppose that \( \alpha \in \mathbb{F}_q^* \) is a \(b\)-linear translator of \( \mathrm{Tr}(H(x)) \) with \( b \neq 0 \). Then

(a) If \( \alpha \mathbb{F}_q \) is uniformly distributed over \( I_0, I_1, \ldots, I_{q-1} \), then \( \alpha \mathbb{F}_q \setminus \{ -\alpha b^{-1} \} \subset P_H \).

(b) If \( \alpha \mathbb{F}_q \) is not uniformly distributed over \( I_0, I_1, \ldots, I_{q-1} \), then \( \alpha \mathbb{F}_q \cap P_H =\{0\}\).
\end{theorem}

\begin{proof}
Part (a). For any \( u, v \in \mathbb{F}_q \) such that \( \alpha u \in I_i \) and \( \alpha v \in I_j \) with \( i \neq j \), it follows that \( \mathrm{Tr}(H(\alpha u)) \neq \mathrm{Tr}(H(\alpha v)) \). Consequently, \( \gamma  \mathrm{Tr}(H(\alpha u)) \neq \gamma  \mathrm{Tr}(H(\alpha v)) \).

Assume that \( \alpha u + \gamma  \mathrm{Tr}(H(\alpha u)) = \alpha v + \gamma  \mathrm{Tr}(H(\alpha v)) \). Rearranging gives
\[\begin{split}
\alpha(u - v) &= \gamma \left( \mathrm{Tr}(H(\alpha v)) - \mathrm{Tr}(H(\alpha u)) \right) \\
&= \gamma \left[ \mathrm{Tr}(H(0+\alpha v)) - \mathrm{Tr}(H(0))+ \mathrm{Tr}(H(0))- \mathrm{Tr}(H(0+\alpha u))  \right]\\
&= \gamma b(v - u).
\end{split}\]
Clearly, the equation above holds if and only if \( \gamma = -\alpha b^{-1} \). This implies that \( \alpha \mathbb{F}_q \setminus \{ -\alpha b^{-1} \} \subset P_H \).

Part (b). There exist distinct elements \( u, v \in \mathbb{F}_q \) such that \( \alpha u \) and \( \alpha v \) belong to the same set \( I_b \), that is, they satisfy \( \mathrm{Tr}(H(\alpha u)) = \mathrm{Tr}(H(\alpha v)) \).
From this, it follows that
\[\mathrm{Tr}(H(0 + \alpha u)) - \mathrm{Tr}(H(0)) = \mathrm{Tr}(H(0 + \alpha v)) - \mathrm{Tr}(H(0)),\]
which implies \( bu = bv \), since $\alpha$ is a \(b\)-linear translator of \( \mathrm{Tr}(H(x)) \). We thus have \( u = v \) because of \( b \neq 0 \). However, this contradicts the fact that \( u \) and \( v \) are distinct elements.
\qed\end{proof}

Denote the set of all linear translators of \(\mathrm{Tr}(H(x))\) as \(\Lambda^*\). Let \(\Lambda = \Lambda^* \cup \{0\}\). Then, \(\Lambda\) is an $\mathbb{F}_q$-linear space over \(\mathbb{F}_{q^n}\) (see \cite{Lai}). Suppose the dimension of \(\Lambda\) is \(m\). Take \(\{\alpha_1, \alpha_2, \dots, \alpha_m\}\) as a basis for \(\Lambda\), where \(\alpha_i\) is a \(b_i\)-linear translator of \(\mathrm{Tr}(H(x))\) for \(i = 1, \dots, m\). By Theorem 2.4, we have known that if $\alpha_i$ is a $0$-linear translator of $\mathrm{Tr}(H(x))$, for all $i=1,2,\cdots,m$, then $\alpha_i \mathbb{F}_q \subseteq {P}_H$. Thus $\alpha_1 \mathbb{F}_q + \cdots + \alpha_m \mathbb{F}_q \subseteq {P}_H$.

We now consider that not all $\alpha_i$ are $0$-linear translators. Without loss of generality, assume that $\alpha_1$ is a $b_1$-linear translator with $b_1 \neq 0$. Then we have the following:

\begin{theorem}
For any \(\alpha_i\in\Lambda\), \(\alpha_1b_i - \alpha_ib_1\) is a \(0\)-linear translator of \(\mathrm{Tr}(H(x))\). Moreover, the set \(\{\alpha_1b_i - \alpha_ib_1 : i = 2, \dots, m\}\) generates an \((m - 1)\)-dimensional \(\mathbb{F}_q\)-linear subspace consisting of \(0\)-linear translators.
\end{theorem}

\begin{proof}
On the one hand, for any \(\alpha\in\Lambda\), we have
\[\begin{split}
&~\mathrm{Tr}(H(x+u(\alpha_1b_i - \alpha_ib_1)))-\mathrm{Tr}(H(x))\\
=&~\mathrm{Tr}(H(x+ub_i\alpha_1 - ub_1\alpha_i)))-\mathrm{Tr}(H(x+ub_i\alpha_1))+\mathrm{Tr}(H(x+ub_i\alpha_1))-\mathrm{Tr}(H(x))\\
=&-ub_1b_i+ub_ib_1=0.
\end{split}\]
From this, we can conclude that \(\alpha_1b_i - \alpha_ib_1\) is a \(0\)-linear translator of \(\mathrm{Tr}(H(x))\).

On the other hand, assume that
\[k_2(\alpha_1b_2 - \alpha_2b_1)+k_3(\alpha_1b_3 - \alpha_3b_1)+\cdots + k_m(\alpha_1b_m - \alpha_mb_1)=0,\]
that is,
\[(k_2b_2 + k_3b_3+\cdots + k_mb_m)\alpha_1 - k_2b_1\alpha_2 - \cdots - k_mb_1\alpha_m = 0.\]
Since \(\alpha_1, \dots, \alpha_m\) form a basis, they are linearly independent. Thus, we can obtain that
\[k_2b_2 + k_3b_3+\cdots + k_mb_m = 0,~-k_2b_1 = 0,\cdots,-k_mb_1 = 0.\]
Also, because \(b_1 \neq 0\), we have \(k_2 = \cdots = k_m = 0\). This indicates that the set \(\{\alpha_1b_i - \alpha_ib_1 : i = 2, \dots, m\}\) is linearly independent. Thus, the set \(\{\alpha_1b_i - \alpha_ib_1 : i = 2, \dots, m\}\) generates an \((m - 1)\)-dimensional \(\mathbb{F}_q\)-linear subspace consisting of \(0\)-linear translators.
\qed\end{proof}

Denote the linear subspace generated by \(\{\alpha_1b_i - \alpha_ib_1 : i = 2, \dots, m\}\) as \(\Lambda_{0}=\langle\alpha_1b_i - \alpha_ib_1 : i = 2, \dots, m\rangle=\langle\beta_i : i = 2, \dots, m\rangle\), where \(\beta_i = \alpha_1b_i - \alpha_ib_1\).

\begin{theorem}
Suppose that \(\alpha\in\mathbb{F}^*_{q^n}\) is a \(b\)-linear translator of \(\mathrm{Tr}(H(x))\), where \(b\neq0\). Then, all elements in the set \(\alpha + \Lambda_{0}\) are \(b\)-linear translators of \(\mathrm{Tr}(H(x))\). Moreover, the set \(\{\alpha, \alpha + \beta_2, \dots, \alpha + \beta_m\}\) forms a basis for \(\Lambda\).
\end{theorem}

\begin{proof}
We only prove the second part. Suppose that
\[k_1\alpha + k_2(\alpha + \beta_2)+\cdots + k_m(\alpha + \beta_m)=0.\]
By rearranging the terms, we get
\[(k_1 + k_2+\cdots + k_m)\alpha=-k_2\beta_2-\cdots - k_m\beta_m .\]
Since \(\alpha\) is not a \(b\)-linear translator ($b\neq0$), it cannot be linearly represented by \(\beta_2,\dots,\beta_m\). Therefore, it must be the case that
\[(k_1 + k_2+\cdots + k_m)\alpha=-k_2\beta_2-\cdots - k_m\beta_m = 0.\]
Moreover, because \(\beta_2,\dots,\beta_m\) are linearly independent, we have
\[-k_2 = 0,\cdots,-k_m = 0.\]
Considering that \(\alpha\in\mathbb{F}_{q^n}^*\), we can conclude that \(k_1 = 0\). In conclusion, \(k_1=\cdots = k_m = 0\), which indicates that \(\{\alpha,\alpha + \beta_2,\dots,\alpha + \beta_m\}\) is linearly independent and thus forms a basis for \(\Lambda\).
\qed\end{proof}

\subsection{The cardinality of the set $P_H$}
Let \(A\) be a set, and let \(|A|\) denote the cardinality of set \(A\). We provide an effective upper bound for the cardinality of the set $P_H$ and show that the upper bound can reach up to $q^n - q^{n - 1}$. Furthermore, we prove that when the cardinality of the set \(P_H\) reaches this upper bound, the function \(\mathrm{Tr}(H(x))\) must be an \(\mathbb{F}_q\)-linear function. We need the following lemmas. The first lemma is due to Ball \cite{Ball}, and the second is due to Yuan \cite{Yuan3}.

Recall that \( D_f \) denotes the direction set determined by the function \( f \).
\begin{lemma}[\hspace{-0.01em}\cite{Ball}]
Let \( p \) be a prime and let \( q = p^n \). Let \( f : \mathbb{F}_q \to \mathbb{F}_q \) be a function such that \( f(0) = 0 \). If \( |D_f| \leq \frac{q + 1}{2} \), then \( f \) is a linearized polynomial.
\end{lemma}

\begin{lemma}[\hspace{-0.01em}\cite{Yuan3}]
Let \( f(x) \in \mathbb{F}_{q^n}[x] \) be a map from \( \mathbb{F}_{q^n} \) to \( \mathbb{F}_q \). Then \( f(x) \) is a \( q \)-polynomial if and only if \( f(x) = \mathrm{Tr}(ux) \) for some \( u \in \mathbb{F}_{q^n} \).
\end{lemma}

\begin{theorem}
Assume that \(\mathrm{Tr}(H(x))\) is not a constant map. The cardinality \(|P_H|\) of the set \(P_H\) satisfies \(|P_H| \leq q^n-q^{n-1}\). This equality holds if and only if \(\mathrm{Tr}(H(x))\) is an \(\mathbb{F}_q\)-linear function. This means that there exists an element \(u\in\mathbb{F}_{q^n}\) such that \(\mathrm{Tr}(H(x))=\mathrm{Tr}(ux)\).
\end{theorem}

\begin{proof}
First, we define
\[P_H(b,c)=\left\{\frac{y - x}{b - c}:x\in I_b,y\in I_c\right\}=\frac{1}{b - c}(I_c - I_b).\]
The number of distinct values that \(\mathrm{Tr}(H(x))\) can take in \(\mathbb{F}_{q^n}\) is at most \(q\), so \(|S|\leq q\). By the pigeonhole principle, among the sets \(I_b\) (\(b\in S\)), there exists at least one \(b_0\in S\) such that \(|I_{b_0}|\geq q^{n - 1}\). This means that
\[\bigg|\bigcup_{\substack{b,c\in S\\b\neq c}}P_H(b,c)\bigg|\geq|P_H(b_0,c)|\geq q^{n - 1}.\]
Since the set \(P_H\) is the complement of \(\bigcup_{\substack{b,c\in S\\b\neq c}}P_H(b,c)\) in \(\mathbb{F}_{q^n}\), we have \[|P_H|=|\mathbb{F}_{q^n}|-\bigg|\bigcup_{\substack{b,c\in S\\b\neq c}}P_H(b,c)\bigg|\leq q^n - q^{n - 1}.\]

As outlined above, there exists a set \(P_H(b_0,c)\) whose number of elements is at least \(q^{n - 1}\). Thus if the number of elements in the set \(\bigcup_{\substack{b,c\in S\\b\neq c}}P_H(b,c)\) is to be \(q^{n - 1}\), then for any \(b\) and \(c\), the number of elements in both the set \(I_b\) and the set \(P_H(b,c)\) is exactly \(q^{n - 1}\). Moreover, for any \(b_1\), \(c_1\), \(b_2\) and \(c_2\), the set \(P_H(b_1,c_1)\) is completely identical to the set \(P_H(b_2,c_2)\). We will prove that when \(|P_H| = q^n - q^{n - 1}\), the function \(\mathrm{Tr}(H(x))\) must be an \(\mathbb{F}_q\)-linear function. Thus, by Lemma 2.2, there exists an element \(u\in\mathbb{F}_{q^n}\) such that \(\mathrm{Tr}(H(x))=\mathrm{Tr}(ux)\).

Without loss of generality, suppose that \(0\in I_0\). From this, we can obtain that \(I_c\subset I_c - I_0\). Also, since \(|I_c| = |P_H(0,c)| = |I_c - I_0| = q^{n - 1}\), it follows that \(I_c = I_c - I_0\). Furthermore, we have \(I_c = I_c+I_0 = I_c+\langle I_0\rangle\). From this, we can see that \(I_0=\langle I_0\rangle\), which indicates that \(I_0\) is a subgroup of order \(q^{n - 1}\), and \(I_c\) is a certain coset of \(I_0\), which can be expressed as \(I_c = g_c+I_0\), where \(g_c\) satisfies \(\mathrm{Tr}(H(g_c)) = c\).

At this point, for any \(b\) and \(c\), we have \(P_H(b,c)=\frac{I_c - I_b}{b-c}=\frac{g_c-g_b}{b- c}+I_0\). We need to show that for any \(b\neq c\), there exists a unique \(\lambda\) such that \(\frac{g_c-g_b}{b-c}=\lambda\). Then \(g_c=-\lambda c+g'_b\), where $g'_b=g_b+\lambda b$ and $\mathrm{Tr}(H(g'_b))=0$.

Assume that \(x_1\in I_{b_1}\) and \(x_2\in I_{b_2}\). For any \(k_1,k_2\in\mathbb{F}_q\), we have \(x_1 = -\lambda b_1+g'_{b_1}\) and \(x_2 = -\lambda b_2+g'_{b_2}\). Then \(k_1x_1 + k_2x_2=-\lambda(k_1b_1 + k_2b_2)+k_1g'_{b_1}+k_2g'_{b_2}\). Therefore,
\[\mathrm{Tr}(H(k_1x_1 + k_2x_2))=k_1b_1 + k_2b_2=k_1\mathrm{Tr}(H(x_1)) + k_2\mathrm{Tr}(H(x_2)).\]
This means that \(\mathrm{Tr}(H(x))\) is an \(\mathbb{F}_q\)-linear function. Thus, by Lemma 2.2, there exists an element \(u\in\mathbb{F}_{q^n}\) such that \(\mathrm{Tr}(H(x))=\mathrm{Tr}(ux)\).
\qed\end{proof}

The result of Theorem 2.8 shows that the maximum cardinality of the set \(P_H\) can reach \(q^n - q^{n - 1}\), and in this case, \(\mathrm{Tr}(H(x))\) must be an \(\mathbb{F}_q\)-linear function.

\begin{theorem} Let \(q\) be an odd prime number. If the cardinality of the set \(P_H\) satisfies \(|P_H| \geq \frac{q^n+1}{2}\), then \(\mathrm{Tr}(H(x))\) is an \(\mathbb{F}_q\)-linear function. That is, there exists an element \(u \in \mathbb{F}_{q^n}\) such that \(\mathrm{Tr}(H(x)) = \mathrm{Tr}(ux)\).
\end{theorem}

\begin{proof}
Let \( D_f \) be the direction set determined by the function \( f \). By Lemma 2.1, if \( |D_{\mathrm{Tr}(H(x))}| \leq \frac{q^n + 1}{2} \), then \( \mathrm{Tr}(H(x)) \) must be a linearized polynomial. Moreover, since \(|P_H| + |D_{\mathrm{Tr}(H(x))}| = |\mathbb{F}_{q^n}|+1\), it follows that \(\mathrm{Tr}(H(x))\) is a linear function whenever \(|P_H| \geq \frac{q^n+1}{2}\). This means that \(\mathrm{Tr}(H(x))\) is an \(\mathbb{F}_q\)-linear function. Thus, by Lemma 2.2, there exists an element \(u\in\mathbb{F}_{q^n}\) such that \(\mathrm{Tr}(H(x))=\mathrm{Tr}(ux)\).
\qed\end{proof}

Combining Theorems 2.8 and 2.9, we conclude that there does not exist a \(H(x)\) such that \(\frac{q^n+1}{2} \leq |P_H| < q^n - q^{n-1}\). In other words, the cardinality of the set \( P_H \) cannot take values within this range. A natural question thus arises: what is the second largest possible value of \( |P_H| \)? Therefore, we pose the following problem:

\begin{problem}
What is the maximum value of the cardinality of the set \( P_H \) that is smaller than \( q^n - q^{n-1} \)?
\end{problem}

In the following examples, we construct some special functions \( \mathrm{Tr}(H(x)) \) and determine the size of the corresponding sets \( P_H \). According to Lagrange interpolation, the existence of \(\mathrm{Tr}(H(x)) \) can be determined by its preimages.

\begin{example}
Let \( q = p \geq 5 \) be a prime. Then there exists a function \( \mathrm{Tr}(H(x)) \) whose values are uniformly distributed over \( \mathbb{F}_q \), and the corresponding set \( P_H = \{0\} \).
\end{example}

\begin{proof} We directly construct the preimage set of the function $\mathrm{Tr}(H(x))$ to illustrate the existence of such $\mathrm{Tr}(H(x))$. View \(\mathbb{F}_{q^n}\) as an \(n\)-dimensional vector space over \(\mathbb{F}_q\), and let \(\{\alpha_1, \dots, \alpha_n\}\) be a basis where \(\alpha_1 = 1\). Then
\[\mathbb{F}_{q^n}=\alpha_1\mathbb{F}_q+\alpha_2\mathbb{F}_q+\cdots+\alpha_{n-1}\mathbb{F}_q+\alpha_{n}\mathbb{F}_q.\]
Let
\[A_{n-1}=\alpha_1\mathbb{F}_q+\alpha_2\mathbb{F}_q+\cdots+\alpha_{n-1}\mathbb{F}_q,\]
we thus have
\[\mathbb{F}_{q^n}=\bigcup_{x\in \mathbb{F}_{q}}\{x\alpha_n+A_{n-1}\}.\]
Recall that $I_b = \{ x \in \mathbb{F}_{q^n} : \mathrm{Tr}(H(x)) = b \}$. We construct the preimage set of $\text{Tr}(H(x))$ in the following manner:
\[\begin{split}
I_{c_i}&=b_i\alpha_n+\left(A_{n-1}\setminus\{1\}\right)\cup\{b_{i+1}\alpha_n+1\},~0\leq i\leq q-1,
\end{split}\]
where $c_0=b_0=b_q=0$, $c_i\in \mathbb{F}_q$ for $0\leq i\leq q-1$, $b_j\in \mathbb{F}_q$ for $0\leq j\leq q-1$, and \(c_i \neq c_j\) and \(b_i \neq b_j\) whenever \(i \neq j\). It is evident that for any \( c_i \neq c_j \), we have \( |I_{c_i}| = |I_{c_j}| = q^{n-1} \). This implies that the preimages of \( \mathrm{Tr}(H(x)) \) constructed in this manner are uniformly distributed over \( \mathbb{F}_q \). Then
\[\begin{split}
\frac{I_{c_i}-I_{c_j}}{c_j-c_i}=&\left\{\frac{b_i-b_j}{c_j-c_i}\alpha_n+A_{n-1}\right\}\bigcup \left\{\frac{b_i-b_{j+1}}{c_j-c_i}\alpha_n+\left(A_{n-1}\setminus\{0\}\right)\right\}\\
&\bigcup \left\{\frac{b_{i+1}-b_{j}}{c_j-c_i}\alpha_n+\left(A_{n-1}\setminus\{0\}\right)\right\}\bigcup \left\{\frac{b_{i+1}-b_{j+1}}{c_j-c_i}\alpha_n\right\},
\end{split}\]
where $1\leq i,j\leq q-1$ and $i\neq j$. If we set $b_i=-c_i,1\leq i\leq q-1,$ then
\[\begin{split}
\frac{I_{c_i}-I_{c_j}}{c_j-c_i}=&\left\{\alpha_n+A_{n-1}\right\}\bigcup \left\{\frac{c_{j+1}-c_i}{c_j-c_i}\alpha_n+\left(A_{n-1}\setminus\{0\}\right)\right\}\\
&\bigcup \left\{\frac{c_{j}-c_{i+1}}{c_j-c_i}\alpha_n+\left(A_{n-1}\setminus\{0\}\right)\right\}\bigcup \left\{\frac{c_{j+1}-c_{i+1}}{c_j-c_i}\alpha_n\right\},
\end{split}\]
where $1\leq i,j\leq q-1$ and $i\neq j$. When $j=0$, we obtain
\[\begin{split}
\frac{I_{c_i}-I_{c_0}}{c_0-c_i}=&\left\{\alpha_n+A_{n-1}\right\}\bigcup \left\{\left(1-\frac{c_1}{c_i}\right)\alpha_n+\left(A_{n-1}\setminus\{0\}\right)\right\}\\
&\bigcup \left\{\frac{c_{i+1}}{c_i}\alpha_n+\left(A_{n-1}\setminus\{0\}\right)\right\}\bigcup \left\{\frac{c_{i+1}-c_{1}}{c_i}\alpha_n\right\}.
\end{split}\]
Since \(1 \leq i \leq q-1\), the expression \(1 - \dfrac{c_1}{c_i}\) can take \(q-1\) distinct values. Moreover, as \(c_1 \neq 0\), it follows that \(1 - \dfrac{c_1}{c_i} \neq 1\). Therefore, \(1 - \dfrac{c_1}{c_i}\) ranges over the set \(\mathbb{F}_q \setminus \{1\}\). Consequently, \(\left(1-\frac{c_1}{c_i}\right)\alpha_n+\left(A_{n-1}\setminus\{0\}\right)\)
ranges over the set $\bigcup_{x\in \mathbb{F}_q \setminus \{1\}}\{x\alpha_n+\left(A_{n-1}\setminus\{0\}\right)\}$. We thus have
\[\bigcup_{1\leq i\leq q-1}\frac{I_{c_i}-I_{c_0}}{c_0-c_i}\supset\bigcup_{x\in\mathbb{F}_q}\{x\alpha_n+\left(A_{n-1}\setminus\{0\}\right)\}\bigcup\{\alpha_n\}=F_{q^n}\setminus\bigcup_{x\in\mathbb{F}_q\setminus \{1\}}\{x\alpha_n\}.\]
From the inclusion
\[\bigcup_{\substack{1\leq i,j\leq q-1\\i\neq j}}\frac{I_{c_i}-I_{c_j}}{c_j-c_i}\supset\bigcup_{1\leq i\leq q-1}\frac{I_{c_i}-I_{c_0}}{c_0-c_i}\bigcup_{\substack{1\leq i,j\leq q-1\\i\neq j}}\left\{\frac{c_{j+1}-c_{i+1}}{c_j-c_i}\alpha_n\right\},\]
if \(\frac{c_{j+1}-c_{i+1}}{c_j-c_i}\) ranges over \(\mathbb{F}_q \setminus \{1\}\), then \(\frac{I_{c_i}-I_{c_j}}{c_j-c_i}\) ranges over \(\mathbb{F}_{q^n}\setminus\{0\}\).

When considering the case where \( q = p \geq 5 \), we define \( c_i = i \) for \( 0 \leq i \leq p-3 \), and set \( c_{p-2} = p-1 \), \( c_{p-1} = p-2 \). It follows that
\[\begin{split}
\bigcup_{\substack{1\leq i,j\leq q-1\\i\neq j}}\left\{\frac{c_{j+1}-c_{i+1}}{c_j-c_i}\right\}=&\bigcup_{3\leq i<j\leq q-1}\left\{\frac{c_{p-i}-c_{p-j}}{c_{p-i-1}-c_{p-j-1}}\right\}\bigcup_{3\leq j\leq q-1}\left\{\frac{c_{p-2}-c_{p-j}}{c_{p-3}-c_{p-j-1}}\right\}\\
&\bigcup\left\{\frac{c_{p-1}-c_{p-2}}{c_{p-2}-c_{p-3}}\right\}\bigcup_{3\leq j\leq q-1}\left\{\frac{c_{p-1}-c_{p-j}}{c_{p-2}-c_{p-j-1}}\right\}\\
=&\bigcup\left\{1\right\}\bigcup_{3\leq j\leq q-1}\left\{1+\frac{1}{j-2}\right\}\bigcup\left\{\frac{p-1}{2}\right\}\bigcup_{3\leq j\leq q-1}\left\{1-\frac{2}{j}\right\}.
\end{split}\]
It is easy to see that \( \bigcup_{3\leq j\leq q-1}\left\{1+\frac{1}{j-2}\right\} = \mathbb{F}_p \setminus \left\{0, 1, \frac{p+1}{2}\right\} \) and \(\bigcup_{3\leq j\leq q-1}\left\{1-\frac{2}{j}\right\} = \mathbb{F}_p \setminus \{0, 1, p-1\} \). Hence, \( \bigcup_{\substack{1\leq i,j\leq q-1\\i\neq j}}\left\{\frac{c_{j+1}-c_{i+1}}{c_j-c_i}\right\} = \mathbb{F}_p \setminus \{0\} \). This implies that \( \bigcup_{\substack{1\leq i,j\leq q-1\\i\neq j}}\frac{I_{c_i}-I_{c_j}}{c_j-c_i} = \mathbb{F}_{p^n} \setminus \{0\} \). Therefore, by the definition of \( P_H \), we have \( P_H = \{0\} \).
\qed\end{proof}

\begin{example}
There exists a function \( \mathrm{Tr}(H(x)) \) that takes only two values, and the cardinality of the corresponding set \( P_H \) is \( q^t \), where \( 0 \leq t \leq n \).
\end{example}

\begin{proof}
Let \(I_b\) and \(I_c\) be two disjoint preimage sets of \(\mathrm{Tr}(H(x))\), where \(I_b\) is assumed to be an additive subgroup of \(\mathbb{F}_{q^n}\) of order \(q^t,0\leq t\leq n\). Then we obtain \(I_c = \mathbb{F}_{q^n} \setminus I_b\) and \(I_b + I_c = I_c\). Furthermore, we have
\[\left|\frac{I_c-I_b}{b-c}\right|=|I_c-I_b|=|I_c+I_b|=|I_c|=q^n-q^{t}.\]
This implies that $|P_H|=q^t$, $0 \leq t \leq n$.
\qed\end{proof}

\section{Permutation polynomials of the form $x+\gamma \mathrm{Tr}(H(x))$}

In this section, we study two classes of functions $H(x)$ over \(\mathbb{F}_{q^2}\) and determine the corresponding sets $P_H$. The sizes of these sets $P_H$ are all relatively small, even only including the trivial case.

We need some knowledge of permutation polynomials in several indeterminates. Let $n\geq1$ and let $\mathbb{F}_q[x_1,\ldots,x_n]$ be the ring of polynomials in $n$ indeterminates over $\mathbb{F}_q$.

\begin{definition}[\hspace{-0.01em}\cite{Lidl-Niederreiter1}]
A polynomial $f\in\mathbb{F}_q[x_1,\ldots,x_n]$ is called a permutation polynomial in $n$ indeterminates over $\mathbb{F}_q$ if the equation $f(x_1,\ldots,x_n)=a$ has $q^{n - 1}$ solutions in $\mathbb{F}_q^n$ for each $a\in\mathbb{F}_q$.
\end{definition}

\begin{definition}[\hspace{-0.01em}\cite{Lidl-Niederreiter1}]
A system of polynomials
\[f_1,\ldots,f_m\in\mathbb{F}_q[x_1,\ldots,x_n],\quad 1\leq m\leq n,\]
is said to be {\it orthogonal} in $\mathbb{F}_q$ if the system of equations
\[f_1(x_1,\ldots,x_n)=a_1,\ldots,f_m(x_1,\ldots,x_n)=a_m\]
has $q^{n - m}$ solutions in $\mathbb{F}_q^n$ for each $(a_1,\ldots,a_m)\in\mathbb{F}_q^m$.
\end{definition}

In the special case $m = n$ this means that the orthogonal system $f_1,\ldots,f_n$ induces a permutation of $\mathbb{F}_q^n$. We could as well say that $f$ is a permutation polynomial if $f$ alone forms an orthogonal system.

\begin{theorem}[{\bf Hermite's Criterion},\hspace{-0.01em}\cite{Lidl-Niederreiter1}]
Let $\mathbb{F}_q$ be of characteristic $p$. Then the system $f_1,\ldots,f_n \in \mathbb{F}_q[x_1,\ldots,x_n]$ is orthogonal in $\mathbb{F}_q$ if and only if the following two conditions are satisfied:

(i) in the reduction of
\[f_1^{q - 1}\cdots f_n^{q - 1}\bmod(x_1^q - x_1,\ldots,x_n^q - x_n)\]
the coefficient of $x_1^{q - 1}\cdots x_n^{q - 1}$ is $\neq 0$;

(ii) in the reduction of
\[f_1^{t_1}\cdots f_n^{t_n}\bmod(x_1^q - x_1,\ldots,x_n^q - x_n)\]
the coefficient of $x_1^{q - 1}\cdots x_n^{q - 1}$ is $0$ whenever $t_1,\ldots,t_n$ are integers with $0\leq t_i\leq q - 1$ for $1\leq i\leq n$, not all $t_i = q - 1$, and at least one $t_i\not\equiv 0\bmod p$.
\end{theorem}

Recall that $\mathbb{F}_{q^2}$ is a two-dimensional vector space over $\mathbb{F}_q$. Let $u$ be a quadratic non-residue and suppose $u = \alpha^2$. Then the set $\{1, \alpha\}$ forms a basis for $\mathbb{F}_{q^2}$ over $\mathbb{F}_q$. For any $x \in \mathbb{F}_{q^2}$ and $\gamma \in \mathbb{F}_{q^2}$, we can express $x = x_1 + x_2\alpha$ and $\gamma = a_1 + a_2\alpha$. It should be notice that $\alpha^q=-\alpha$ by the definition of $\alpha$ and thus $\mathrm{Tr}(\alpha) = 0$. In each proof of the following theorems, we will consistently use this basis. Let $\text{char}(\mathbb{F}_{q^2})$ be the characteristic of $\mathbb{F}_{q^2}$.

\subsection{The case $x+\gamma \mathrm{Tr}(x^k)$}

\begin{theorem}
Let \(q\) be an odd prime power. The polynomial \(f(x)=x+\gamma \mathrm{Tr}(x^2)\) over \(\mathbb{F}_{q^2}\) is a permutation polynomial if and only if \(\gamma = 0\). This implies that the set \(P_H = \{0\}\).
\end{theorem}

\begin{proof}
Let \( x = x_1 + x_2\alpha \) and \( \gamma = a_1 + a_2\alpha \), with \( \mathrm{Tr}(\alpha) = 0 \) and \( \alpha^2 = u \), where \( u \) is a quadratic non-residue. We have
\[\begin{split}
f(x)&=f(x_1 + x_2\alpha)\\
&=x_1 + x_2\alpha+(a_1 + a_2\alpha)\mathrm{Tr}((x_1 + x_2\alpha)^2)\\
&=x_1 + x_2\alpha+(a_1 + a_2\alpha)\mathrm{Tr}(x_1^2 + ux_2^2+2\alpha x_1x_2)\\
&=x_1 + x_2\alpha+(a_1 + a_2\alpha)(2x_1^2 + 2ux_2^2)\\
&=(2a_1x_1^2+2a_1ux_2^2+x_1)+(2a_2x_1^2+2a_2ux_2^2+x_2)\alpha.
\end{split}\]
Let
\[f_1(x_1,x_2)= 2a_1x_1^2+2a_1ux_2^2+x_1,\quad f_2(x_1,x_2) = 2a_2x_1^2+2a_2ux_2^2+x_2.\]
Then
\[f(x)=f_1(x_1,x_2)+ f_2(x_1,x_2)\alpha.\]
Recall that \(f(x)\) is a permutation polynomial if and only if \(f_1(x_1, x_2)\) and \(f_2(x_1, x_2)\) are orthogonal.

We first consider the case where \(a_1a_2\neq0\). According to Theorem 3.1, if there exist \(0\leq t_1,t_2\leq q - 1\), where \(t_1\) and \(t_2\) are not both equal to \(q - 1\), such that the coefficient of \(x_1^{q - 1}x_2^{q - 1}\) is not equal to 0, then \(f(x)\) must not be a permutation polynomial. In fact, when \(t_1=t_2=\frac{q - 1}{2}\), we have
\[\begin{split}
f^{t_1}_1(x_1,x_2)f^{t_2}_2(x_1,x_2)&=f^{\frac{q - 1}{2}}_1(x_1,x_2)f^{\frac{q - 1}{2}}_2(x_1,x_2)\\
&=(2a_1x_1^2+2a_1ux_2^2+x_1)^{\frac{q - 1}{2}}(2a_2x_1^2+2a_2ux_2^2+x_2)^{\frac{q - 1}{2}}\\
&=2^{q-1}(a_1a_2)^{\frac{q - 1}{2}}\left(x_1^2+ux_2^2+\frac{x_1}{2a_1}\right)^{\frac{q - 1}{2}}\left(x_1^2+ux_2^2+\frac{x_2}{2a_2}\right)^{\frac{q - 1}{2}}\\
&=2^{q-1}(a_1a_2)^{\frac{q - 1}{2}}\left(x_1^2+ux_2^2\right)^{q - 1}+\text{the terms without}~x_1^{q - 1}x_2^{q - 1}\\
&=2^{q-1}(a_1a_2)^{\frac{q - 1}{2}}u^{\frac{q - 1}{2}}\binom{q-1}{\frac{q-1}{2}}x_1^{q - 1}x_2^{q - 1}+\text{the terms without}~x_1^{q - 1}x_2^{q - 1}.
\end{split}\]
Then the coefficient of \(x_1^{q - 1}x_2^{q - 1}\) is \(2^{q-1}(a_1a_2)^{\frac{q - 1}{2}}u^{\frac{q - 1}{2}}\binom{q-1}{\frac{q-1}{2}}\), and
\[u^{\frac{q - 1}{2}}\binom{q-1}{\frac{q-1}{2}} \equiv(-1)^{\frac{q+1}{2}}\pmod{q}.\]
Thus, when \(a_1a_2\neq0\), by Theorem 3.1, \(f_1(x_1, x_2)\) and \(f_2(x_1, x_2)\) are not orthogonal. Hence, \(f(x)\) is not a permutation polynomial.

When \(a_1 = a_2 = 0\), that is, \(\gamma = 0\), the function \(f(x)=x\) is a permutation polynomial. When either \(a_1 = 0\) or \(a_2 = 0\), \(f(x)\) is not a permutation polynomial. For instance, when \(a_1 = 0\) and \(a_2\neq0\), we get
\[f_1(x_1,x_2)=x_1,\quad f_2(x_1,x_2) = 2a_2x_1^2+2a_2ux_2^2+x_2.\]
The polynomial \( f(x) \) is a permutation polynomial if and only if the system of equations \( f_1(x_1,x_2)= r \) and \( f_2(x_1,x_2) = s \) has exactly one solution over \( \mathbb{F}_q^2 \) for each \( (r, s) \in \mathbb{F}_q^2 \). If \( a_2 \neq 0 \), then for any \( r\in \mathbb{F}_q\setminus\{0\} \), the system \( (f_1, f_2) = (r, 2a_2r^2) \) admits two solutions \( (x_1, x_2) = (r, 0) \) and \( \left(r, -\frac{1}{2a_2r}\right) \). Therefore, the polynomial \(f(x)=x+\gamma \mathrm{Tr}(x^2)\) over \(\mathbb{F}_{q^2}\) is a permutation polynomial if and only if \(\gamma = 0\).
\qed\end{proof}

\begin{theorem}
Let \(q\) be an odd prime power. The polynomial \(f(x)=x+\gamma \mathrm{Tr}\left(x^{q+1}\right)\) over \(\mathbb{F}_{q^2}\) is a permutation polynomial if and only if \(\gamma = 0\).
\end{theorem}

\begin{proof}
Set \( x = x_1 + x_2\alpha \) and \( \gamma = a_1 + a_2\alpha \), with \( \mathrm{Tr}(\alpha) = 0 \) and \( \alpha^2 = u \). We have
\[\begin{split}
f(x)&=f(x_1 + x_2\alpha)\\
&=x_1 + x_2\alpha+(a_1 + a_2\alpha)\mathrm{Tr}\left((x_1 + x_2\alpha)^{q+1}\right)\\
&=x_1 + x_2\alpha+(a_1 + a_2\alpha)\mathrm{Tr}\left((x_1 - x_2\alpha)(x_1 + x_2\alpha)\right)\\
&=x_1 + x_2\alpha+(a_1 + a_2\alpha)\mathrm{Tr}\left(x_1^2 - x_2^2u\right)\\
&=x_1 + x_2\alpha+(a_1 + a_2\alpha)(2x_1^2 - 2x_2^2u)\\
&=(x_1+2a_1x_1^2-2a_1x_2^2u)+(x_2+2a_2x_1^2-2a_2x_2^2u)\alpha.
\end{split}\]
Let
\[f_1(x_1,x_2)=x_1+2a_1x_1^2-2a_1x_2^2u,\quad f_2(x_1,x_2)=x_2+2a_2x_1^2-2a_2x_2^2u.\]
Take $t_1=\frac{q-1}{2},t_2=\frac{q-1}{2}$, we have
\[\begin{split}
f^{t_1}_1(x_1,x_2)f^{t_2}_2(x_1,x_2)&=f^{\frac{q-1}{2}}_1(x_1,x_2)f^{\frac{q-1}{2}}_2(x_1,x_2)\\
&=(x_1+2a_1x_1^2-2a_1x_2^2u)^{\frac{q-1}{2}}(x_2+2a_2x_1^2-2a_2x_2^2u)^{\frac{q-1}{2}}\\
&=2^{q-1}(a_1a_2)^{\frac{q-1}{2}}(x_1^2-x_2^2u)^{q-1}+\text{the terms without}~x_1^{q - 1}x_2^{q - 1}\\
&=2^{q-1}(a_1a_2)^{\frac{q-1}{2}}\binom{q-1}{\frac{q-1}{2}}u^{\frac{q - 1}{2}}x_1^{q - 1}x_2^{q - 1}+\text{the terms without}~x_1^{q - 1}x_2^{q - 1}.\\
\end{split}\]
Thus, when \(a_1a_2\neq0\), by Theorem 3.1, \(f_1(x_1, x_2)\) and \(f_2(x_1, x_2)\) are not orthogonal. Hence, \(f(x)\) is not a permutation polynomial. When either \(a_1 = 0\) or \(a_2 = 0\), \(f(x)\) is not a permutation polynomial. For instance, when \(a_1 = 0\) and \(a_2\neq0\), we get
\[f_1(x_1,x_2)=x_1,\quad f_2(x_1,x_2) = x_2+2a_2x_1^2-2a_2x_2^2u.\]
If \( a_2 \neq 0 \), then for any \( r\in \mathbb{F}_q\setminus\{0\} \), the system \( (f_1, f_2) = (r, 2a_2r^2) \) admits two solutions \( (x_1, x_2) = (r, 0) \) and \( \left(r, \frac{1}{2a_2u}\right) \). Therefore, the polynomial \(f(x)=x+\gamma \mathrm{Tr}(x^{q+1})\) over \(\mathbb{F}_{q^2}\) is a permutation polynomial if and only if \(\gamma = 0\).
\qed\end{proof}

In 2016, Kyureghyan and Zieve \cite{Kyureghyan-Zieve} showed that when \( q = 7 \) and \( \gamma^4 = 1 \), the polynomial \(f(x)=x+\gamma \mathrm{Tr}\left(x^{q+3}\right)\) over \(\mathbb{F}_{q^2}\) is a permutation polynomial. Here, \( q = 7 \) satisfies \( q \equiv 3 \pmod{4} \). We proceed to prove the following result.

\begin{theorem}
Let \(q\) be an odd prime power such that \( q \equiv 1 \pmod{4} \). The polynomial \(f(x)=x+\gamma \mathrm{Tr}\left(x^{q+3}\right)\) over \(\mathbb{F}_{q^2}\) is a permutation polynomial if and only if \(\gamma = 0\).
\end{theorem}

\begin{proof}
Let \( x = x_1 + x_2\alpha \) and \( \gamma = a_1 + a_2\alpha \), with \( \mathrm{Tr}(\alpha) = 0 \) and \( \alpha^2 = u \). We have
\[\begin{split}
f(x)&=f(x_1 + x_2\alpha)\\
&=x_1 + x_2\alpha+(a_1 + a_2\alpha)\mathrm{Tr}\left((x_1 + x_2\alpha)^{q+3}\right)\\
&=x_1 + x_2\alpha+(a_1 + a_2\alpha)\mathrm{Tr}\left((x_1 - x_2\alpha)(x_1 + x_2\alpha)^3\right)\\
&=x_1 + x_2\alpha+(a_1 + a_2\alpha)\mathrm{Tr}\left(x_1^4+2x_1^3x_2\alpha-2x_1x_2^3u\alpha-x_2^4u^2\right)\\
&=x_1 + x_2\alpha+(a_1 + a_2\alpha)(2x_1^4-2x_2^4u^2)\\
&=(x_1+2a_1x_1^4-2a_1x_2^4u^2)+(x_2+2a_2x_1^4-2a_2x_2^4u^2)\alpha.
\end{split}\]
Let
\[f_1(x_1,x_2)=x_1+2a_1x_1^4-2a_1x_2^4u^2,\quad f_2(x_1,x_2)=x_2+2a_2x_1^4-2a_2x_2^4u^2.\]
Take $t_1=\frac{q-1}{2},t_2=0$, we have
\[\begin{split}
f^{t_1}_1(x_1,x_2)f^{t_2}_2(x_1,x_2)&=f^{\frac{q-1}{2}}_1(x_1,x_2)f^{0}_2(x_1,x_2)\\
&=(x_1+2a_1x_1^4-2a_1x_2^4u^2)^{\frac{q-1}{2}}\\
&=\binom{\frac{q-1}{2}}{\frac{q-1}{4}}2^{\frac{q-1}{2}}a_1^{\frac{q-1}{2}}u^{\frac{q - 1}{2}}x_1^{q - 1}x_2^{q - 1}+\text{the terms without}~x_1^{q - 1}x_2^{q - 1}.\\
\end{split}\]
Thus, if \( f(x)\) is a permutation polynomial, then it must hold that \( a_1 = 0 \). Furthermore,
\[\begin{split}
f^{0}_1(x_1,x_2)f^{\frac{q-1}{2}}_2(x_1,x_2)&=(x_2+2a_2x_1^4-2a_2x_2^4u^2)^{\frac{q-1}{2}}\\
&=(-4)^{\frac{q - 1}{4}}\binom{\frac{q-1}{2}}{\frac{q-1}{4}}u^{\frac{q - 1}{2}}a_2^{\frac{q - 1}{2}}x_1^{q - 1}x_2^{q - 1}+\text{the terms without}~x_1^{q - 1}x_2^{q - 1}.
\end{split}\]
Hence, we get \(a_2=0\). Therefore, the polynomial \(f(x)=x+\gamma \mathrm{Tr}(x^{q+3})\) over \(\mathbb{F}_{q^2}\) is a permutation polynomial if and only if \(\gamma = 0\).
\qed\end{proof}

\subsection{The case $x+\gamma \mathrm{Tr}(x^{k+1}\pm x^{q+k})$}
In 2019, Zha et al. \cite{Zha-Hu-Zhang} obtained the following results:

(i) Let \( n \) and \( k \) be positive integers, \( q = p^k \), and \( \gamma \in \mathbb{F}_q^* \). The polynomial \( f(x) = x + \gamma \mathrm{Tr}_q^{q^n}(x^2 - x^{q + 1}) \) permutes \( \mathbb{F}_{q^n} \).

(ii) Let \( k \) be a positive integer, and let \( q = p^k \) where $p$ is an odd prime. Assume \( \gamma \in \mathbb{F}_q^* \) and \( -2\gamma \) is a square of \( \mathbb{F}_q^* \). The polynomial \( f(x) = x + \gamma \mathrm{Tr}_q^{q^2}(x^3 - x^{q + 2}) \) permutes \( \mathbb{F}_{q^2} \).

For result (i), we give a new proof by using Theorem 2.4 and generalize the result to \( x + \gamma\mathrm{Tr}_q^{q^n}(x^2 - x^{q^i + 1}) \), where \( 1 \leq i \leq n-1 \). For the case of \( \mathbb{F}_{q^2} \), we provide an alternative proof for these two results and show that the sufficient condition is also necessary.
\begin{theorem}
(i) If \(\gamma\in \mathbb{F}_q\), then the polynomial \( f(x) = x + \gamma \mathrm{Tr}_q^{q^n}(x^2 - x^{q^i + 1}) \) permutes \( \mathbb{F}_{q^n} \), where \( 1 \leq i \leq n-1 \).

(ii) Let \(q\) be an odd prime power. The polynomial \(f(x)=x+\gamma \mathrm{Tr}_q^{q^2}\left(x^2-x^{q+1}\right)\) over \(\mathbb{F}_{q^2}\) is a permutation polynomial if and only if \(\gamma\in \mathbb{F}_q\).
\end{theorem}

\begin{proof}
(i) By Theorem 2.4, we only need to prove that $1$ is a \( 0 \)-linear translator of \( \mathrm{Tr}_q^{q^n}(x^2 - x^{q^i + 1}) \). For any $r\in \mathbb{F}_q$, we have
\[\begin{split}
\mathrm{Tr}_q^{q^n}\left((x+r)^2 - (x+r)^{q^i + 1}\right)&=\mathrm{Tr}_q^{q^n}\left(x^2+2rx+r^2 - (x+r)^{q^i}(x+r)\right)\\
&=\mathrm{Tr}_q^{q^n}\left(x^2+2rx+r^2 - (x^{q^i+1}+rx^{q^i}+rx+r^2)\right)\\
&=\mathrm{Tr}_q^{q^n}\left((x^2-x^{q^i+1})+r(x-x^{q^i})\right)\\
&=\mathrm{Tr}_q^{q^n}(x^2 - x^{q^i + 1}).
\end{split}\]
Hence, $1$ is a \( 0 \)-linear translator of \( \mathrm{Tr}_q^{q^n}(x^2 - x^{q^i + 1}) \). By Theorem 2.4, it follows that $\mathbb{F}_q\subset P_H$.

(ii) Set \( x = x_1 + x_2\alpha \) and \( \gamma = a_1 + a_2\alpha \), with \( \mathrm{Tr}(\alpha) = 0 \) and \( \alpha^2 = u \). We have
\[\begin{split}
f(x)&=f(x_1 + x_2\alpha)\\
&=x_1 + x_2\alpha+(a_1 + a_2\alpha)\mathrm{Tr}_q^{q^2}\left((x_1 + x_2\alpha)^2-(x_1 + x_2\alpha)^{q+1}\right)\\
&=x_1 + x_2\alpha+(a_1 + a_2\alpha)\mathrm{Tr}_q^{q^2}\left((x_1 + x_2\alpha)^2-(x_1 - x_2\alpha)(x_1 + x_2\alpha)\right)\\
&=x_1 + x_2\alpha+(a_1 + a_2\alpha)\mathrm{Tr}_q^{q^2}\left(2\alpha x_2(x_1 + x_2\alpha)\right)\\
&=x_1 + x_2\alpha+4ux_2^2(a_1 + a_2\alpha)\\
&=(x_1+4a_1ux_2^2)+x_2(1+4a_2ux_2)\alpha.
\end{split}\]
Let
\[f_1(x_1,x_2)=x_1+4a_1ux_2^2,\quad f_2(x_1,x_2)=x_2(1+4a_2ux_2).\]
If \( a_2 \neq 0 \), then for any \( r\in \mathbb{F}_q\setminus\{\frac{a_1}{4ua_2^2}\} \), the system \( (f_1, f_2) = (r,0) \) admits two solutions \( (x_1, x_2) = (r, 0) \) and \( \left(r, r-\frac{a_1}{4ua_2^2}\right) \). Hence, \( a_2 \) must equal $0$. At this time, for any \( a_1\in \mathbb{F}_q \), \( f_1(x_1,x_2) \) and \( f_2(x_1,x_2) \) are orthogonal, so \( f(x) \) is a permutation polynomial.
\qed\end{proof}

Note that a similar result can be obtained from Theorem 3.5 (ii) with minor adjustments as shown below:
\begin{corollary}
Let \(q\) be an odd prime power. The polynomial \(f(x)=x+\gamma \mathrm{Tr}\left(x^2+x^{q+1}\right)\) over \(\mathbb{F}_{q^2}\) is a permutation polynomial if and only if \(\alpha\gamma\in \mathbb{F}_q\), where $\alpha^2=u$ and $u$ is a quadratic non-residue.
\end{corollary}

\begin{proof}
Let $x=\alpha t$, $t\in\mathbb{F}_{q^2}$, then we have
\[f(x)=\alpha t+\gamma \mathrm{Tr}\left((\alpha t)^2+(\alpha t)^{q+1}\right)=\alpha t+\gamma \mathrm{Tr}\left(ut^2-u t^{q+1}\right)=\alpha \left(t+\alpha\gamma \mathrm{Tr}\left(t^2-t^{q+1}\right)\right).\]
Therefore, by Theorem 3.5, \(f(x)=x+\gamma \mathrm{Tr}\left(x^2+x^{q+1}\right)\) over \(\mathbb{F}_{q^2}\) is a permutation polynomial if and only if \(\alpha\gamma\in \mathbb{F}_q\).
\end{proof}

\begin{theorem}
Let \(q\) be an odd prime power. The polynomial \(f(x)=x+\gamma \mathrm{Tr}\left(x^3-x^{q+2}\right)\) over \(\mathbb{F}_{q^2}\) is a permutation polynomial if and only if $\gamma\in\mathbb{F}_q$ and \(-2\gamma \) is a square of $\mathbb{F}_q$.
\end{theorem}

\begin{proof}
Let \( x = x_1 + x_2\alpha \) and \( \gamma = a_1 + a_2\alpha \), with \( \mathrm{Tr}(\alpha) = 0 \) and \( \alpha^2 = u \). We have
\[\begin{split}
f(x)&=f(x_1 + x_2\alpha)\\
&=x_1 + x_2\alpha+(a_1 + a_2\alpha)\mathrm{Tr}\left((x_1 + x_2\alpha)^3-(x_1 + x_2\alpha)^{q+2}\right)\\
&=x_1 + x_2\alpha+(a_1 + a_2\alpha)\mathrm{Tr}\left(2x_1^2x_2\alpha+2ux_2^3\alpha+4x_1x_2^2u\right)\\
&=x_1 + x_2\alpha+8x_1x_2^2u(a_1 + a_2\alpha)\\
&=x_1(1+8a_1x_2^2u)+x_2(1+8a_2x_1x_2u)\alpha.
\end{split}\]
Let
\[f_1(x_1,x_2)=x_1(1+8a_1x_2^2u),\quad f_2(x_1,x_2)=x_2(1+8a_2x_1x_2u).\]
When \(a_1 = a_2 = 0\), i.e., \(\gamma = 0\), \(f(x)\) is a permutation polynomial. When \(a_1 = 0\) and \(a_2 \neq 0\), for any \(r \in \mathbb{F}_q \setminus \{0\}\), the system \( (f_1, f_2) = (r,0) \) admits two solutions \((x_1, x_2) = (r, 0)\) and \((r, \frac{-1}{8a_2ru})\); hence, \(f(x)\) is not a permutation polynomial. When \(a_2 = 0\) and \(a_1 \neq 0\), \(f(x)\) is a permutation polynomial if and only if \(1 + 8a_1x_2^2u\) never equals zero. This implies that $\gamma\in\mathbb{F}_q$ and \( -2a_1 = -2\gamma \) is a square of $\mathbb{F}_q$.

Next, we discuss the case where \(a_1 \neq 0\) and \(a_2 \neq 0\). If there exists \(t \in \mathbb{F}_q\setminus\{0\}\) such that \(1 + 8a_1t^2u = 0\), for any \(s \in \mathbb{F}_q \setminus \{t\}\), the system \((f_1, f_2)\) has two distinct solutions \((x_1, x_2) = (0, s)\) and \((t, \frac{s-t}{8a_2ut^2})\). If \(1 + 8a_1x_2^2u\) never equals zero, we show that there exists \(r \in \mathbb{F}_q\setminus\{0\}\) such that \((f_1, f_2) = (r, 0)\) has two distinct solutions \((x_1, x_2)\), which indicates that \(f(x)\) is not a permutation polynomial. From \(f_2 = 0\), we obtain \(x_2 = 0\) or \(x_2 = \frac{-1}{8a_2x_1u}\). If \(x_2 = 0\), then from \(f_1 = r\), we get \(x_1 = r\). If \(x_2 = \frac{-1}{8a_2x_1u}\), substituting into \(f_1 = r\) yields \(x_1 + \frac{a_1}{8a_2^2ux_1} = r\), where \(x_1 \neq 0\) and \(x_1 \neq r\). After simplification, we derive \(x_1^2 - rx_1 + \frac{a_1}{8a_2^2u} = 0\). If there exists \(r \in \mathbb{F}_q \setminus \{0\}\) such that the discriminant \(\Delta = r^2 - \frac{a_1}{2a_2^2u}\) of this quadratic equation is a square, then the quadratic equation has solutions. Let us set \(\Delta = r^2 - \frac{a_1}{2a_2^2u} = w^2\), so \((r + w)(r - w) = \frac{a_1}{2a_2^2u}\). Let \(r + w = a_1k\) and \(r - w = \dfrac{1}{2a_2^2uk}\) with \(k \in \mathbb{F}_q \setminus \{0\}\), then we get \(r = \frac{a_1k}{2}+\frac{1}{4a_2^2uk}\) and \(w = \frac{a_1k}{2}-\frac{1}{4a_2^2uk}\). Therefore, in this case, \(f(x)\) is not a permutation polynomial. To sum up, \( f(x) \) is a permutation polynomial if and only if $\gamma\in\mathbb{F}_q$ and \(-2\gamma \) is a square of $\mathbb{F}_q$.
\qed\end{proof}

\begin{theorem}
Let \(q\) be an odd prime power. If $\text{char}(\mathbb{F}_{q^2})\neq3$, the polynomial \(f(x)=x+\gamma \mathrm{Tr}\left(x^4-x^{q+3}\right)\) over \(\mathbb{F}_{q^2}\) is a permutation polynomial if and only if \(\gamma = 0\). If $\text{char}(\mathbb{F}_{q^2})=3$, the polynomial \(f(x)=x+\gamma \mathrm{Tr}\left(x^4-x^{q+3}\right)\) over \(\mathbb{F}_{q^2}\) is a permutation polynomial if and only if \(\gamma\in \mathbb{F}_q\).
\end{theorem}

\begin{proof} Suppose that $\text{char}(\mathbb{F}_{q^2})\neq3$. Set \( x = x_1 + x_2\alpha \) and \( \gamma = a_1 + a_2\alpha \), with \( \mathrm{Tr}(\alpha) = 0 \) and \( \alpha^2 = u \). We have
\[\begin{split}
f(x)&=f(x_1 + x_2\alpha)\\
&=x_1 + x_2\alpha+(a_1 + a_2\alpha)\mathrm{Tr}\left((x_1 + x_2\alpha)^4-(x_1 + x_2\alpha)^{q+3}\right)\\
&=x_1 + x_2\alpha+(a_1 + a_2\alpha)\mathrm{Tr}\left((x_1 + x_2\alpha)^4-(x_1 - x_2\alpha)(x_1 + x_2\alpha)^3\right)\\
&=x_1 + x_2\alpha+(a_1 + a_2\alpha)\mathrm{Tr}\left(2x_1^3x_2\alpha+6x_1^2x_2^2u+6x_1x_2^2u\alpha+2x_2^4u^2\right)\\
&=x_1 + x_2\alpha+(a_1 + a_2\alpha)\left(12x_1^2x_2^2u+4x_2^4u^2\right)\\
&=(x_1+12a_1x_1^2x_2^2u+4a_1x_2^4u^2)+(x_2+12a_2x_1^2x_2^2u+4a_2x_2^4u^2)\alpha.
\end{split}\]
Let
\[f_1(x_1,x_2)=x_1+12a_1x_1^2x_2^2u+4a_1x_2^4u^2,\quad f_2(x_1,x_2)=x_2+12a_2x_1^2x_2^2u+4a_2x_2^4u^2.\]
Take $t_1=\frac{q-1}{2},t_2=0$, we have
\[\begin{split}
f^{t_1}_1(x_1,x_2)f^{t_2}_2(x_1,x_2)&=f^{\frac{q-1}{2}}_1(x_1,x_2)f^{0}_2(x_1,x_2)\\
&=(x_1+12a_1x_1^2x_2^2u+4a_1x_2^4u^2)^{\frac{q-1}{2}}\\
&=12^{\frac{q - 1}{2}}a_1^{\frac{q - 1}{2}}u^{\frac{q - 1}{2}}x_1^{q - 1}x_2^{q - 1}+\text{the terms without}~x_1^{q - 1}x_2^{q - 1}.
\end{split}\]
Thus, if \( f(x)\) is a permutation polynomial, then it must hold that \( a_1 = 0 \). Furthermore,
\[\begin{split}
f^{0}_1(x_1,x_2)f^{\frac{q-1}{2}}_2(x_1,x_2)&=(x_2+12a_2x_1^2x_2^2u+4a_2x_2^4u^2)^{\frac{q-1}{2}}\\
&=12^{\frac{q - 1}{2}}a_2^{\frac{q - 1}{2}}u^{\frac{q - 1}{2}}x_1^{q - 1}x_2^{q - 1}+\text{the terms without}~x_1^{q - 1}x_2^{q - 1}.
\end{split}\]
Hence, we get \(a_2=0\). Therefore, the polynomial \(f(x)=x+\gamma \mathrm{Tr}\left(x^4-x^{q+3}\right)\) over \(\mathbb{F}_q^2\) is a permutation polynomial if and only if \(\gamma = 0\).

If $\text{char}(\mathbb{F}_{q^2})=3$, we have
\[f_1(x_1,x_2)=x_1+4a_1x_2^4u^2,\quad f_2(x_1,x_2)=x_2+4a_2x_2^4u^2.\]
Since \( (3, q^2 - 1) = 1 \), \( x^3 \) is a permutation polynomial over \( \mathbb{F}_{q^2} \). Then for any \( a_2 \neq 0 \), there exists \( t \) such that \( 1 + 4a_2u^2t^3 = 0 \). If \( a_2 \neq 0 \), then for any \( r\in \mathbb{F}_q \), the system \( (f_1, f_2) = (r, 0) \) admits two solutions \( (x_1, x_2) = (r, 0) \) and \( \left(r-4a_1u^2t^4,t\right) \). Hence, \( a_2 \) must equal $0$. Moreover, $f_1(x_1,x_2)$ is a permutation polynomial for any $a_1\in \mathbb{F}_q$.
\qed\end{proof}

\begin{corollary}
Let \(q\) be an odd prime power. If $\text{char}(\mathbb{F}_{q^2})\neq3$, the polynomial \(f(x)=x+\gamma \mathrm{Tr}\left(x^4+x^{q+3}\right)\) over \(\mathbb{F}_{q^2}\) is a permutation polynomial if and only if \(\gamma = 0\). If $\text{char}(\mathbb{F}_{q^2})=3$, the polynomial \(f(x)=x+\gamma \mathrm{Tr}\left(x^4+x^{q+3}\right)\) over \(\mathbb{F}_{q^2}\) is a permutation polynomial if and only if \(\alpha\gamma\in \mathbb{F}_q\), where $\alpha^2=u$ and $u$ is a quadratic non-residue.
\end{corollary}

\begin{proof}
Let $x=\alpha t$, $t\in\mathbb{F}_{q^2}$, then we have
\[f(x)=\alpha t+\gamma \mathrm{Tr}\left((\alpha t)^4+(\alpha t)^{q+3}\right)=\alpha t+\gamma \mathrm{Tr}\left(u^2t^4-u^2 t^{q+3}\right)=\alpha \left(t+\alpha u\gamma \mathrm{Tr}\left(t^4-t^{q+3}\right)\right).\]
If $\text{char}(\mathbb{F}_{q^2})\neq3$, by Theorem 3.7, \(f(x)=x+\gamma \mathrm{Tr}\left(x^4+x^{q+3}\right)\) over \(\mathbb{F}_{q^2}\) is a permutation polynomial if and only if \(\alpha u\gamma=0\), i.e., \(\gamma=0\). If $\text{char}(\mathbb{F}_{q^2})=3$, the polynomial \(f(x)=x+\gamma \mathrm{Tr}\left(x^4+x^{q+3}\right)\) over \(\mathbb{F}_{q^2}\) is a permutation polynomial if and only if \(\alpha u\gamma\in \mathbb{F}_q\), i.e., \(\alpha\gamma\in \mathbb{F}_q\).
\end{proof}

\begin{theorem}
Let \(q\) be an odd prime power. The polynomial \(f(x)=x+\gamma \mathrm{Tr}\left(x^{q+3}-x^{2q+2}\right)\) over \(\mathbb{F}_{q^2}\) is a permutation polynomial if and only if \(\gamma = 0\).
\end{theorem}

\begin{proof}
Set \( x = x_1 + x_2\alpha \) and \( \gamma = a_1 + a_2\alpha \), with \( \mathrm{Tr}(\alpha) = 0 \) and \( \alpha^2 = u \). We have
\[\begin{split}
f(x)&=f(x_1 + x_2\alpha)\\
&=x_1 + x_2\alpha+(a_1 + a_2\alpha)\mathrm{Tr}\left((x_1 + x_2\alpha)^{q+3}-(x_1 + x_2\alpha)^{2q+2}\right)\\
&=x_1 + x_2\alpha+(a_1 + a_2\alpha)\mathrm{Tr}\left((x_1 - x_2\alpha)(x_1 + x_2\alpha)^3-(x_1 - x_2\alpha)^{2}(x_1 + x_2\alpha)^2\right)\\
&=x_1 + x_2\alpha+(a_1 + a_2\alpha)\mathrm{Tr}\left(2x_1^3x_2\alpha+2x_1^2x_2^2u-2x_1x_2^3u\alpha-2x_2^4u^2\right)\\
&=x_1 + x_2\alpha+(a_1 + a_2\alpha)\left(4x_1^2x_2^2u-4x_2^4u^2\right)\\
&=(x_1+4a_1x_1^2x_2^2u-4a_1x_2^4u^2)+(x_2+4a_2x_1^2x_2^2u-4a_2x_2^4u^2)\alpha.
\end{split}\]
Let
\[f_1(x_1,x_2)=x_1+4a_1x_1^2x_2^2u-4a_1x_2^4u^2,\quad f_2(x_1,x_2)=x_2+4a_2x_1^2x_2^2u-4a_2x_2^4u^2.\]
Take $t_1=\frac{q-1}{2},t_2=0$, we have
\[\begin{split}
f^{t_1}_1(x_1,x_2)f^{t_2}_2(x_1,x_2)&=f^{\frac{q - 1}{2}}_1(x_1,x_2)f^{0}_2(x_1,x_2)\\
&=(x_1+4a_1x_1^2x_2^2u-4a_1x_2^4u^2)^{\frac{q - 1}{2}}\\
&=4^{\frac{q - 1}{2}}a_1^{\frac{q - 1}{2}}u^{\frac{q - 1}{2}}x_1^{q - 1}x_2^{q - 1}+\text{the terms without}~x_1^{q - 1}x_2^{q - 1}.
\end{split}\]
Thus, if \( f(x)\) is a permutation polynomial, then it must hold that \( a_1 = 0 \). Furthermore,
\[\begin{split}
f^{0}_1(x_1,x_2)f^{\frac{q - 1}{2}}_2(x_1,x_2)&=(x_2+4a_2x_1^2x_2^2u-4a_2x_2^4u^2)^{\frac{q - 1}{2}}\\
&=4^{\frac{q - 1}{2}}a_2^{\frac{q - 1}{2}}u^{\frac{q - 1}{2}}x_1^{q - 1}x_2^{q - 1}+\text{the terms without}~x_1^{q - 1}x_2^{q - 1}.
\end{split}\]
Hence, we get \(a_2=0\). Therefore, the polynomial \(f(x)=x+\gamma \mathrm{Tr}\left(x^{q+3}-x^{2q+2}\right)\) over \(\mathbb{F}_{q^2}\) is a permutation polynomial if and only if \(\gamma = 0\).
\qed\end{proof}

\begin{corollary}
Let \(q\) be an odd prime power. The polynomial \(f(x)=x+\gamma \mathrm{Tr}\left(x^{q+3}+x^{2q+2}\right)\) over \(\mathbb{F}_{q^2}\) is a permutation polynomial if and only if \(\gamma=0\).
\end{corollary}

\begin{proof}
Let $x=\alpha t$, $t\in\mathbb{F}_{q^2}$, then we have
\[f(x)=\alpha t+\gamma \mathrm{Tr}\left((\alpha t)^{q+3}+(\alpha t)^{2q+2}\right)=\alpha \left(t-\alpha u\gamma \mathrm{Tr}\left(t^{q+3}-t^{2q+2}\right)\right).\]
Therefore, by Theorem 3.8, \(f(x)=x+\gamma \mathrm{Tr}\left(x^{q+3}+x^{2q+2}\right)\) over \(\mathbb{F}_{q^2}\) is a permutation polynomial if and only if \(-\alpha u\gamma=0\)., i.e., \(\gamma=0\).
\end{proof}

\section{Conclusion}
In this paper, we focused on the set
\[P_H=\{\gamma\in \mathbb{F}_{q^n} : x+\gamma \mathrm{Tr}(H(x))~\text{is a permutation polynomial}\}.\]
On the one hand, we characterized the structure of the set \(P_H\) and revealed its connection with the direction set \(D_{\mathrm{Tr}(H(x))}\). On the other hand, in Theorems 2.4 and 2.5, we established the relationship between the set \(P_H\) and the linear translators of the function \(\mathrm{Tr}(H(x))\).

Additionally, in Theorem 2.8, we derived an effective upper bound for the cardinality \(|P_H|\) of the set \(P_H\), and proved that when \(|P_H|\) reached the maximum value, the function \(\mathrm{Tr}(H(x))\) must be an \(\mathbb{F}_q\)-linear function. It was further shown that the value of \(|P_H|\) did not fall within the interval \(\left[\frac{q^n + 1}{2}, q^n - q^{n-1}\right)\). This naturally led to a new question: what is the second largest value of \(|P_H|\)? We also constructed a function \(\mathrm{Tr}(H(x))\) that takes values uniformly over \(\mathbb{F}_q\), for which the corresponding set \(P_H\) contains only trivial cases. Finally, for two classes of functions \(H(x)\), we discussed the necessary and sufficient conditions for \(x + \gamma \mathrm{Tr}(H(x))\) to be a permutation polynomial over \(\mathbb{F}_{q^2}\), and determined the corresponding set \(P_H\) by using Hermite's Criterion.


\begin{thebibliography}{}

\bibitem{Ball} S. Ball, \emph{The number of directions determined by a function over a finite field}, J. Combin. Theory Ser. A 104 (2003), 341--350.

\bibitem{Blokhuis-Ball-Brouwer-Storme-Szonyi} A. Blokhuis, S. Ball, A. E. Brouwer, L. Storme, and T. Sz\H{o}nyi, \emph{On the number of slopes of the graph of a function defined on a finite field}, J. Combin. Theory Ser. A 86 (1999), 187--196.

\bibitem{Blokhuis-Brouwer-Szonyi} A. Blokhuis, A. E. Brouwer and T. Sz\H{o}nyi, \emph{The number of directions determined by a function $f$ on a finite field}, J. Combin. Theory Ser. A 70 (1995), 349--353.

\bibitem{Charpin-Kyureghyan0} P. Charpin and G. Kyureghyan, \emph{On a class of permutation polynomials over $\mathbb{F}_{2^n}$}, in: SETA 2008, in: Lecture Notes in Comput. Sci., vol. 5203, Springer-Verlag, 2008, pp. 368--376.

\bibitem{Charpin-Kyureghyan1} P. Charpin and G. Kyureghyan, \emph{When does $G(x)+\gamma \mathrm{Tr}(H(x))$ permute $\mathbb{F}_{p^n}$?}, Finite Fields Appl. 15 (2009), 615--632.

\bibitem{Charpin-Kyureghyan2} P. Charpin and G. Kyureghyan, \emph{Monomial functions with linear structure and permutation polynomials}, Contemp. Math. 518 (2010), 99--111.

\bibitem{Charpin-Kyureghyan-Suder} P. Charpin, G. Kyureghyan and V. Suder, \emph{Sparse permutations with low differential uniformity}, Finite Fields Appl. 28 (2014), 214--243.

\bibitem{Ding} C. Ding, \emph{Cyclic codes from some monomials and trinomials}, SIAM J. Discrete Math. 27 (2013), 1977--1994.

\bibitem{Ding-Yuan} C. Ding and J. Yuan, \emph{A family of skew Hadamard difference sets}, J. Comb. Theory, Ser. A 113 (2006), 1526--1535.

\bibitem{Ding-Zhou} C. Ding and Z. Zhou, \emph{Binary cyclic codes from explicit polynomials over $GF(2^m)$}, Discrete Math. 321 (2014), 76--89.

\bibitem{Evans-Greene-Niederreiter} R. J. Evans, J. Greene and H. Niederreiter, \emph{Linearized polynomials and permutation polynomials of finite fields}, Michigan Math. J. 39 (1992), 405--413.

\bibitem{Gacs-Lovasz-Szonyi} A. G{\'a}cs, L. Lov{\'a}sz and T. Sz\H{o}nyi, \emph{On a generalization of R{\'e}dei's theorem}, Combinatorica 23 (2003), 585--598.

\bibitem{Jiang-Yuan-Li-Qu} S. Jiang, M. Yuan, K. Li and L. Qu, \emph{New constructions of permutation polynomials of the form \(x+\gamma\mathrm{Tr}^{q^2}_q(h(x))\) over finite fields with even characteristic}, Finite Fields Appl. 101 (2025), 102522.

\bibitem{Kyureghyan} G. Kyureghyan, \emph{Constructing permutations of finite fields via linear translators}, J. Combin. Theory Ser. A 118(3) (2011), 1052--1061.

\bibitem{Kyureghyan-Zieve} G. Kyureghyan and M. Zieve, \emph{Permutation polynomials of the form $X+\gamma\mathrm{Tr}(X^k)$}, Contemporary Developments in Finite Fields and Applications, World Scientific, (2016), 178--194.

\bibitem{Lai} X. Lai, \emph{Additive and linear structures of cryptographic functions}, in: Proc. of FSE, in: Lecture Notes in Comput. Sci., 1008 (1995), 75--85.

\bibitem{Laigle-Chapuy} Y. Laigle-Chapuy, \emph{Permutation polynomials and applications to coding theory}, Finite Fields Appl. 13 (2007), 58--70.

\bibitem{Li-Qu-Chen-Li} K. Li, L. Qu, X. Chen and C. Li \emph{Permutation polynomials of the form $cx + \mathrm{Tr}_{q^l/q}(x^a)$ and permutation trinomials over finite fields with even characteristic}, Cryptogr. Commun. 10 (2018), 531--554.

\bibitem{Li-Qu-Wang} K. Li, L. Qu and Q. Wang \emph{Compositional inverses of permutation polynomials of the form $x^rh(x^s)$ over finite fields}, Cryptogr. Commun. 11 (2019), 279--298.

\bibitem{Lidl-Niederreiter1} R. Lidl and H. Niederreiter, \emph{Finite Fields}, Encyclopedia Math. Appl., vol. 20, Addison-Wesley, 1983.

\bibitem{Lidl-Niederreiter2} R. Lidl and H. Niederreiter, \emph{Introduction to Finite Fields and their Applications}, Cambridge: Cambridge University Press, 1986.

\bibitem{Lovasz-Schrijver} L. Lov{\'a}sz and A. Schrijver \emph{Remarks on a theorem of R{\'e}dei}, Studia Scient. Math. Hungar. 16 (1981), 449--454.

\bibitem{Mullen} G. L. Mullen, \emph{Permutation polynomials over finite fields. Proceedings of the Conference on Finite Fields and Their Applications}, Lecture Notes in Pure and Applied Mathematics, Vol. 141. Marcel Dekker, New York, (1993), 131--151.

\bibitem{Niederreiter-Robinson} H. Niederreiter and K.H. Robinson, \emph{Complete mappings of finite fields}, J. Aust. Math. Soc. Ser. A 33 (1982), 197--212.

\bibitem{Polverino-Szonyi-Weiner} O. Polverino, T. Sz\H{o}nyi and Z. Weiner, \emph{Blocking sets in Galois planes of square order}, Acta Sci. Math. (Szeged) 65 (1999), 773--784.

\bibitem{Redei} L. R{\'e}dei, \emph{L{\"u}ckenhafte Polynome {\"u}ber endlichen K{\"o}rpen}, Birkh{\"a}user-Verlag, Basel, 1970. (English translation: Lacunary Polynomials over finite fields, North-Holland, Amsterdam, 1973.)

\bibitem{Rivest-Shamir-Adelman} R. L. Rivest, A. Shamir and L. M. Adelman, \emph{A method for obtaining digital signatures and public-key cryptosystems}, Commun. ACM 21 (1978), 120--126.

\bibitem{Schwenk-Huber} J. Schwenk and K. Huber, \emph{Public key encryption and digital signatures based on permutation polynomials}, Electron. Lett. 34 (1998), 759--760.

\bibitem{Tuxanidy-Wang} A. Tuxanidy and Q. Wang, \emph{Compositional inverses and complete mappings over finite fields}, Discrete Appl. Math. 217 (2017), 318--329.

\bibitem{Wang-Zha-Du-Zheng} Y. Wang, Z. Zha, X. Du and D. Zheng, \emph{Several classes of permutation polynomials with trace functions over $\mathbb{F}_{q^n}$}, Appl. Algebr. Eng. Comm. 35 (2024), 337--349.

\bibitem{Wu-Yuan} D. Wu and P. Yuan, \emph{Permutation polynomials and their compositional inverses over finite fields by a local method}, Des. Codes Cryptogr. 92(2) (2024), 267--276.

\bibitem{Yuan1} P. Yuan, \emph{Compositional inverses of AGW-PPs}, Adv. Math. Commun. 16(4) (2022), 1185--1195.

\bibitem{Yuan2} P. Yuan, \emph{Local method for compositional inverses of permutation polynomials}, Commun. Algebra 52(7) (2022), 3070--3080.

\bibitem{Yuan3} P. Yuan, \emph{Algebraic structure of permutational polynomials over $\mathbb{F}_{q^n}$}, Commun. Algebra (2025), 1--10, https://doi.org/10.1080/00927872.2025.2476727.

\bibitem{Yuan-Ding1} P. Yuan and C. Ding, \emph{Permutation polynomials over finite fields from a powerful lemma}, Finite Fields Appl. 17 (2011), 560--574.

\bibitem{Yuan-Ding2} P. Yuan and C. Ding, \emph{Further results on permutation polynomials over finite fields}, Finite Fields Appl. 27 (2014), 88--103.

\bibitem{Yuan-Zeng} P. Yuan and X. Zeng, \emph{A note on linear permutation polynomials}, Finite Fields Appl. 17 (2011), 488--491.

\bibitem{Zha-Hu-Zhang} Z. Zha, L. Hu and Z. Zhang \emph{Permutation polynomials of the form $x +\gamma \mathrm{Tr}^{q^n}_q(h(x))$}, Finite Fields Appl. 60 (2019), 101573.

\end{thebibliography}
\end{document}